\documentclass{article}
\usepackage{hyperref}
\usepackage[T1]{fontenc}
\usepackage[utf8]{inputenc}
\usepackage{authblk}
\usepackage{tikz}
\usepackage{latexsym,amsfonts,amssymb,epsfig,verbatim}
\usepackage{amsmath,amsthm,amssymb,latexsym,graphics,textcomp}
\usepackage{graphicx}
\usepackage{url}
\usepackage{fullpage}
\usepackage{hyperref}
\hypersetup{
    colorlinks=true,
    linkcolor=blue,
    citecolor=blue,
    filecolor=blue,
    urlcolor=blue,
}

\newtheorem{theorem}{Theorem}
\newtheorem{proposition}{Proposition}

\newtheorem{corollary}{Corollary}

\theoremstyle{definition}
\newtheorem{definition}{Definition}

\newtheorem{Ex}{Example}
\newcommand{\nulll}{\hskip4pt}
\DeclareMathOperator{\gyr}{gyr}
\DeclareMathOperator{\Aut}{Aut}
\DeclareMathOperator{\LCay}{L-Cay}
\DeclareMathOperator{\RCay}{R-Cay}

\begin{document}

\title{On transitivity and connectedness of Cayley graphs of gyrogroups}
\author[1]{Rasimate Maungchang}
\author[1]{Prathomjit Khachorncharoenkul}
\author[1]{Kiattisak Prathom}
\author[2]{Teerapong Suksumran\footnote{Corresponding author.}$^{,}$\,}
\affil[1]{School of Science, Walailak University, Nakhon Si Thammarat 80160, Thailand}
\affil[2]{Department of Mathematics, Faculty of Science, Chiang Mai University, Chiang Mai 50200, Thailand}
\renewcommand\Authands{ and }

\date{\today}
\maketitle

\begin{abstract}
    In this work, we explore edge direction, transitivity, and connectedness of Cayley graphs of gyrogroups. More specifically, we find conditions for a Cayley graph of a gyrogroup to be undirected, transitive, and connected. We  also show a relationship between the cosets of a certain type of subgyrogroups and the connected components of Cayley graphs.  Some examples regarding these findings are provided.
\end{abstract}

\section{Introduction}
\label{sec:Introduction}
Gyrogroups, a generalization of groups whose the associativity is replaced with a more general one, have been one of the fast growing area in Mathematics recently. The structure was introduced by A. A. Ungar in the attempt to find a proper structure to Einstein's velocity model, see~\cite{Ungar} for more details. Since then, many properties, including geometric, algebraic, and topological properties of gyrogroups have been studied. 

For a groups, regarding the combinatorial property,  its Cayley graph is considered as a combinatorial representation of that group, giving a visualization to its algebraic structure. The vertices of a Cayley graph of a group $G$ is the elements of $G$ and there is a directed edge from a vertex $u$ to a vertex $v$, denoted $u\to v,$ if $v=su$ for some $s\in G$. We will call this graph a \textit{left Cayley graph} or an \textit{L-Cayley graph}. Another definition is defined by changing the edge condition to $u\to v$ if $v=us$ for some $s\in G$, giving a \textit{right Cayley graph} or an \textit{R-Cayley graph}.
It is then natural to think about the same combinatorial structures of gyrogroups for they are a generalization of groups. Some preliminary properties and examples of L-Cayley graphs of gyrogroups have been studied in \cite{LAS}. In that study, a property on connectedness of Cayley graphs of gyrogroups has been proved. 

In this study, we further explore  these combinatorial structures of gyrogroups, more precisely, we study properties of transitivity of L-Cayley graphs of gyrogroups, edge direction, transitivity, connectedness and  connected components of  R-Cayley graphs of gyrogroups.

\textbf{Outline of the paper.} We give necessary definitions and background knowledge in Section~\ref{sec:Background}, including the definitions of L-Cayley graphs and R-Cayley graphs of gyrogroups.
In Section~\ref{sec:Transitivity_L} we give some sufficient conditions for an L-Cayley graph of a gyrogroup to be transitive together with an example. Section~\ref{sec:Transitivity_R} is devoted for R-Cayley graphs of gyrogroups. In this section, we give a sufficient and necessary condition of an R-Cayley graph to be undirected, give a sufficient condition for it to be transitive, and show a connection of the cosets of L-subgyrogroups with the connected components of the graph.
A few examples illustrating these results are also given.

\section{Background}
\label{sec:Background}
Included in this section are necessary background. We give  important definitions and algebraic identities regarding gyrogroups as well as the definitions of two types of Cayley graphs of gyrogroups. For more detailed knowledge of gyrogroups, we recommend readers to see~\cite{Ungar}, ~\cite{Ungar1}, and~\cite{Suksumran2016}. 

Let $(G,\oplus)$ be a groupoid. Sometimes, we will simply call it $G$ when there is no possible confusion. An \textit{automorphism} $f$ on $G$ is a bijection from $G$ to itself with the property that $f(g_1\oplus g_2)=f(g_1)\oplus f(g_2)$, for all $g_1,g_2\in G$. The set of all automorphisms on $G$ is denoted by $\Aut(G,\oplus)$.

\begin{definition}[Definition 2.7 of \cite{Ungar}]
\label{gyrogroup}
        Let $(G,\oplus)$ be a nonempty groupoid. We say that $G$  is a \textit{gyrogroup} if the following hold:
        \begin{enumerate}
                \item There is a unique identity element $e \in G$ such that
                \begin{center}
                        $e\oplus x = x = x\oplus e$ \ \ \ \ \ for all $x \in G$.
                \end{center}
                \item For each $x \in G$, there exists a unique \textit{inverse} element $\ominus x \in G$ such that 
                \begin{center}
                        $\ominus x\oplus x = e = x\oplus (\ominus x)$.
                \end{center}
                \item For any $x, y \in G$, there exists $\gyr[x,y] \in \Aut(G,\oplus)$ such that
                \begin{center}
                        $x\oplus(y\oplus z) = (x\oplus y)\oplus \gyr[x,y](z)$
                \end{center}
                for all $z \in G$.\hfill \textnormal{(left gyroassociative law)}
                \item For any $x,y \in G$, $\gyr[x\oplus y,y] = \gyr[x,y]$.\hfill \textnormal{(left loop property)}        
        \end{enumerate}

\end{definition}

\begin{Ex}[Example 8, p. 60 of \cite{Suksumran2016}]
An example of a finite gyrogroup of order $15$ is $G_{15} = \{0, 1, 2,\ldots, 14\}$ whose operation is given by Table \ref{tab: operation G15}. Its gyration table is described by Table \ref{tab: gyration G15}. In cyclic notation, four nonidentity gyroautomorphisms of $G_{15}$ can be expressed as in (\ref{eqn:
gyration of G15}):
\begin{eqnarray}\label{eqn: gyration of G15}
\begin{split}
A &= (1\nulll 7 \nulll 5 \nulll 10 \nulll 6)(2 \nulll 3 \nulll 8
\nulll 11 \nulll 14),\\
B &= (1\nulll 6 \nulll 10 \nulll 5 \nulll 7)(2 \nulll 14 \nulll 11
\nulll 8 \nulll 3),\\
C &= (1\nulll 10 \nulll 7 \nulll 6 \nulll 5)(2 \nulll 11 \nulll 3
\nulll 14 \nulll 8),\\
D &= (1\nulll 5 \nulll 6 \nulll 7 \nulll 10)(2 \nulll 8 \nulll 14
\nulll 3 \nulll 11).
\end{split}
\end{eqnarray}

\begin{table}
\centering
\begin{tabular}{|c|ccccccccccccccc|}
\hline
$\oplus$ & 0 & 1 & 2 & 3 & 4 & 5 & 6 & 7 & 8 & 9 & 10 & 11 & 12 & 13 & 14\\ \hline
0  & 0  & 1  & 2  & 3  & 4  & 5  & 6  & 7  & 8  & 9  & 10 & 11 & 12 & 13 & 14 \\ 
1  & 1  & 2  & 0  & 4  & 6  & 11 & 3  & 14 & 13 & 7  & 8  & 12 & 5  & 10 & 9 \\ 
2  & 2  & 0  & 1  & 6  & 3  & 12 & 4  & 9  & 10 & 14 & 13 & 5  & 11 & 8  & 7 \\ 
3  & 3  & 4  & 5  & 7  & 8  & 9  & 13 & 0  & 1  & 2  & 12 & 6  & 14 & 11 & 10\\ 
4  & 4  & 10 & 8  & 11 & 13 & 1  & 5  & 6  & 14 & 0  & 7  & 2  & 9  & 12 & 3\\ 
5  & 5  & 14 & 12 & 9  & 7  & 8  & 2  & 11 & 0  & 10 & 3  & 4  & 6  & 1  & 13\\ 
6  & 6  & 11 & 4  & 13 & 10 & 3  & 14 & 8  & 12 & 1  & 2  & 9  & 7  & 5  & 0\\ 
7  & 7  & 8  & 9  & 0  & 1  & 2  & 11 & 3  & 4  & 5  & 14 & 13 & 10 & 6  & 12\\ 
8  & 8  & 13 & 6  & 10 & 11 & 0  & 12 & 4  & 5  & 3  & 9  & 7  & 2  & 14 & 1\\ 
9  & 9  & 5  & 11 & 14 & 0  & 6  & 7  & 10 & 2  & 12 & 1  & 3  & 13 & 4  & 8\\ 
10 & 10 & 3  & 13 & 12 & 5  & 14 & 8  & 2  & 9  & 6  & 11 & 0  & 1  &  7 & 4\\ 
11 & 11 & 12 & 7  & 1  & 14 & 4  & 9  & 13 & 6  & 8  & 0  & 10 & 3  &  2 & 5\\ 
12 & 12 & 6  & 3  & 8  & 9  & 7  & 10 & 1  & 11 & 13 & 5  & 14 & 4  & 0  & 2\\ 
13 & 13 & 7  & 14 & 2  & 12 & 10 & 1  & 5  & 3  & 4  & 6  & 8  & 0  & 9  & 11\\ 
14 & 14 & 9  & 10 & 5  & 2  & 13 & 0  & 12 & 7  & 11 & 4  & 1  & 8  & 3  & 6\\ \hline
\end{tabular}
\caption{Addition table for the gyrogroup $G_{15}$.} \label{tab: operation G15}
\end{table}

\begin{table}
\centering
\begin{tabular}{|c|ccccccccccccccc|}
\hline
0 & $I$ & $I$ & $I$ & $I$ & $I$ & $I$ & $I$ & $I$ & $I$ & $I$ & $I$ & $I$ & $I$ & $I$ & $I$ \\ \hline
1 & $I$ & $I$ & $I$ & $A$ & $A$ & $B$ & $C$ & $D$ & $D$ & $B$ & $A$ & $C$ & $C$ & $D$ & $B$ \\ 
2 & $I$ & $I$ & $I$ & $D$ & $B$ & $D$ & $B$ & $A$ & $B$ & $A$ & $C$ & $A$ & $D$ & $C$ & $C$ \\ 
3 & $I$ & $B$ & $C$ & $I$ & $B$ & $A$ & $C$ & $I$ & $D$ & $A$ & $D$ & $B$ & $D$ & $C$ & $A$ \\ 
4 & $I$ & $B$ & $A$ & $A$ & $I$ & $B$ & $B$ & $B$ & $A$ & $I$ & $B$ & $A$ & $I$ & $I$ & $A$ \\ 
5 & $I$ & $A$ & $C$ & $B$ & $A$ & $I$ & $B$ & $C$ & $I$ & $B$ & $D$ & $A$ & $C$ & $D$ & $D$ \\ 
6 & $I$ & $D$ & $A$ & $D$ & $A$ & $A$ & $I$ & $B$ & $C$ & $B$ & $C$ & $B$ & $C$ & $D$ & $I$ \\ 
7 & $I$ & $C$ & $B$ & $I$ & $A$ & $D$ & $A$ & $I$ & $A$ & $B$ & $B$ & $D$ & $C$ & $D$ & $C$ \\ 
8 & $I$ & $C$ & $A$ & $C$ & $B$ & $I$ & $D$ & $B$ & $I$ & $A$ & $A$ & $D$ & $D$ & $C$ & $B$ \\ 
9 & $I$ & $A$ & $B$ & $B$ & $I$ & $A$ & $A$ & $A$ & $B$ & $I$ & $A$ & $B$ & $I$ & $I$ & $B$ \\ 
10 & $I$ & $B$ & $D$ & $C$ & $A$ & $C$ & $D$ & $A$ & $B$ & $B$ & $I$ & $I$ & $C$ & $D$ & $A$ \\
11 & $I$ & $D$ & $B$ & $A$ & $B$ & $B$ & $A$ & $C$ & $C$ & $A$ & $I$ & $I$ & $D$ & $C$ & $D$ \\ 
12 & $I$ & $D$ & $C$ & $C$ & $I$ & $D$ & $D$ & $D$ & $C$ & $I$ & $D$ & $C$ & $I$ & $I$ & $C$ \\ 
13 & $I$ & $C$ & $D$ & $D$ & $I$ & $C$ & $C$ & $C$ & $D$ & $I$ & $C$ & $D$ & $I$ & $I$ & $D$ \\ 
14 & $I$ & $A$ & $D$ & $B$ & $B$ & $C$ & $I$ & $D$ & $A$ & $A$ & $B$ & $C$ & $D$ & $C$ & $I$ \\ \hline
\end{tabular}
\caption{Gyration table for $G_{15}$. Here, $I$ denotes the identity
automorphism of $G_{15}$, $A, B, C$ and $D$ are given by (\ref{eqn:
gyration of G15}).} \label{tab: gyration G15}
\end{table}

With the absence of associativity, the gyrogroup $G_{15}$ is not a group. For any elements $a, b, c$ in $G_{15}$, the gyroautomorphism $\gyr[a,b]$ comes from the identity,  called the gyrator identity, which is true in general for every gyrogroup:
\begin{equation}\tag{\textnormal{gyrator identity}}
\gyr[a,b]c=\ominus(a\oplus b)\oplus(a\oplus(b\oplus c).
\end{equation}
\end{Ex}

In this work we will work on finite gyrogroups. Inspired by the solution of the equation $x\oplus a=b$, Ungar introduced a binary operation in $G$ called the \textit{gyrogroup coaddition} or \textit{coaddition} $\boxplus$, defined by
\[a\boxplus b=a\oplus\gyr[a,\ominus b]b,\]
for all $a,b\in G$. We write $a\boxminus b$ for $a\boxplus \ominus b$. Then the solution to the equation $x\oplus a=b$ is $x=b\boxminus a$.

Many identities regarding the gyrogroup addition and coaddition have been discovered and can be found together with the proofs in~\cite{Ungar}. We list some identities needed later in this work here.  

\begin{theorem}[\cite{Ungar}]
\label{T:identities}
Let $(G,\oplus)$ be a gyrogroup. For any $a,b,c$, the following properties hold:

\begin{enumerate}
\item if $a\oplus b=a\oplus c$, then $b=c$;\hfill \textnormal{(general left cancellation law)} 
\label{Id:general left cancellation law}

\item $\ominus a\oplus(a\oplus b)=b$;\hfill \textnormal{(left cancellation law)}
\label{Id:left cancellation law}

\item $(a\ominus b)\boxplus b=a$;\hfill \textnormal{(right cancellation law I)}
\label{Id:right cancellation law I}

\item $(a\boxminus b)\oplus b=a$;\hfill \textnormal{(right cancellation law II)}
\label{Id:right cancellation law II}

\item $(a\oplus b)\oplus c=a\oplus(b\oplus\gyr[b,a]c)$;\hfill \textnormal{(right gyroassociative law)}
\label{Id:right gyroassociative law}

\item $\gyr[a,b](\ominus c)=\ominus\gyr[a,b]c$.
\label{Id:inverse through gyro}

\end{enumerate}
\end{theorem}

The fourth author has thoroughly studied algebraic properties of gyrogroups analogous to those of groups; among the work, the following definitions and theorems are important to our work. We encourage
readers to see~\cite{Suksumran2016}
for more explanations and motivations.

\begin{definition}
A nonempty subset $H$ of a gyrogroup $(G,\oplus)$ is a\textit{ subgyrogroup} of $G$ if $(H,\oplus)$ is a gyrogroup and $\gyr[a,b](H)=H$ for all $a,b\in H$. It is called an \textit{L-subgyrogroup} of $G$ if $\gyr[a,b](H)=H$ for all $a\in G$ and $h\in H$.
\end{definition}

\begin{theorem}
\label{T:partition}
If $H$ is an L-subgyrogroup of a gyrogroup $G$, then the set $\{g\oplus H\mid g\in G\}$ forms a partition of $G$.
\end{theorem}
From Theorem \ref{T:partition}, when $H$ is an L-subgyrogroup of a gyrogroup $G$, we will call each $g\oplus G$, \textit{a left coset}.

\begin{theorem}[Theorem 21 of \cite{suksumran2015}, Lagrange's Theorem for L-Subgyrogroups]
\label{T:Lagrange Lsub}
If $H$ is an L-subgyrogroup of a finite gyrogroup $G$, then $|H|$ divides $|G|$.
\end{theorem}

Writing $[G:H]$ as the number of left cosets of $H$ in $G$, we have the following corollary as a consequence of Theorem~\ref{T:Lagrange Lsub}.
\begin{corollary}
\label{T:number of cosets}
If $H$ is an L-subgyrogroup of a finite gyrogroup $G$, then $|G|=[G:H]|H|$.
\end{corollary}

In the last part of this section, we turn to a combinatorial representation of a gyrogroup analogous to that of group, a Cayley graph. The following definitions are the Cayley graph version of gyrogroups.
\begin{definition}
\label{d:LRCay}
Let $G$ be a gyrogroup and let $S$ be a subset of $G$ not containing the identity element $e$. The \textit{L-Cayley graph} or \textit{left-Cayley graph of} $G$ \textit{generated by} $S$, denoted by $\LCay(G,S)$, is a directed graph whose vertices are the gyrogroup elements, and for any two vertices $u$ and $v$, $u\to v$ if $v=s\oplus u$ for some $s\in S$. We will conflate the gyrogroup elements and the vertices of graph whenever there are no confusions. 

In the same sense, the \textit{R-Cayley graph} or \textit{right-Cayley graph of} $G$ \textit{generated by} $S$, denoted $\RCay(G,S)$, is a directed graph whose the vertex set is $G$ and, for any two vertices $u$ and $v$, $u\to v$ if $v=u\oplus s$ for some $s\in S$.

If a Cayley graph has a property that $v\to u$ whenever $u\to v$, then we say that the graph is \textit{undirected}. In this case, we may draw each edge with arrows on both ends or drop the arrows entirely.
\end{definition}
If $S$ is the empty set, then each type of Cayley graphs is the union of disjoint vertices, each corresponding to an element of that gyrogroup. We do not allow the identity element $e$ to be in $S$ to avoid the presence of loops in the graph, and from this time forward, we will assume this condition without mentioning it. Many examples of Cayley graphs of gyrogroups are given in succeeding sections.

\section{Left-Cayley graphs}
\label{sec:Transitivity_L}
We begin this section by looking at some theorems and examples regarding L-Cayley graphs of gyrogroups provided in~\cite{LAS} and discuss about the transitivity. After that, we provide some sufficient conditions for an L-Cayley graph of a gyrogroup to be transitive together with an example.
\begin{definition}
Let $S$  be subset of a gyrogroup $G$. The set $S$ is said to be \textit{symmetric} if for each element $s\in S$, $\ominus s\in S$. The \textit{left-generating set by} $S$, written $(S\rangle$, is 
\[
(S\rangle =\{s_n\oplus(\cdots\oplus(s_3\oplus(s_2\oplus s_1))\cdots)\mid s_1\ldots,s_n\in S\}.
\]
If $(S\rangle=G$, we say $S$ \textit{left-generates }$G$, or $G$ is \textit{left-generated} by $S$.  The \textit{right-generating set} is defined in a similar fashion. 
\end{definition}

Two familiar results to the Cayley graphs of groups were proven in~\cite{LAS} in the case of gyrogroups and we restate them here.

\begin{theorem}[Theorem 3.1 in \cite{LAS}]
\label{LAST31}
Let $G$ be a gyrogroup and let $S$ be a subset of $G$. Then, $\LCay(G,S)$ is undirected if and only if $S$ is symmetric.
\end{theorem}
\begin{theorem}[Theorem 3.3 in \cite{LAS}]
Let $G$ be a gyrogroup and let $S$ be a symmetric subset of $G$. Then, $\LCay(G,S)$ is connected if and only if $S$ left-generates $G$.
\end{theorem}

In the following example, we present some L-Cayley graphs of a gyrogroup given in~\cite{LAS} and note some facts that make the transitivity of L-Cayley graphs of gyrogroups different than that of groups.

\begin{Ex}
\label{E:G8}
Let $G_8=\{0,1,2,3,4,5,6,7\}$ be a gyrogroup defined by the addition and gyration tables shown in Table~\ref{Ta:G8}. It was exhibited in~\cite{LAS} that $G$ has both transitive and non-transitive L-Cayley graphs. The subset $\{1,3\}$ is a left-generating set and symmetric. So, the L-Cayley graph $\LCay(G_{8},\{1,3\})$ is connected and undirected. It is a cycle, hence transitive, see Figure~\ref{F:G8S13andG8S123}. Whereas, 
the L-Cayley graph $\LCay(G_{8},\{1,2,3\})$ shown in the same figure is not transitive. We note here that some right additions by an element of $G$ is not an automorphism on an L-Cayley graph. For example, the right addition by $1$ on $\LCay(G_{8},\{1,3\})$ maps two adjacent vertices $4$ and $5$ to non-adjacent vertices $2$ and $4$, respectively. This phenomenon is different from that of Cayley graphs of groups where the right additions are automorphisms on the graphs. The reason is the occurrence of gyration: suppose $v=s\oplus u$, i.e., $u$ and $v$ are adjacent, then adding $g$ on the right to both sides gives $v\oplus g=(s\oplus u)\oplus g=s\oplus(u\oplus\gyr[ u,s](g))$, and we lose adjacency. This observation gives rise to Theorem~\ref{T:Ltransitive}.
\qed

\begin{figure}[ht]
\begin{center}
\includegraphics[height=3cm]{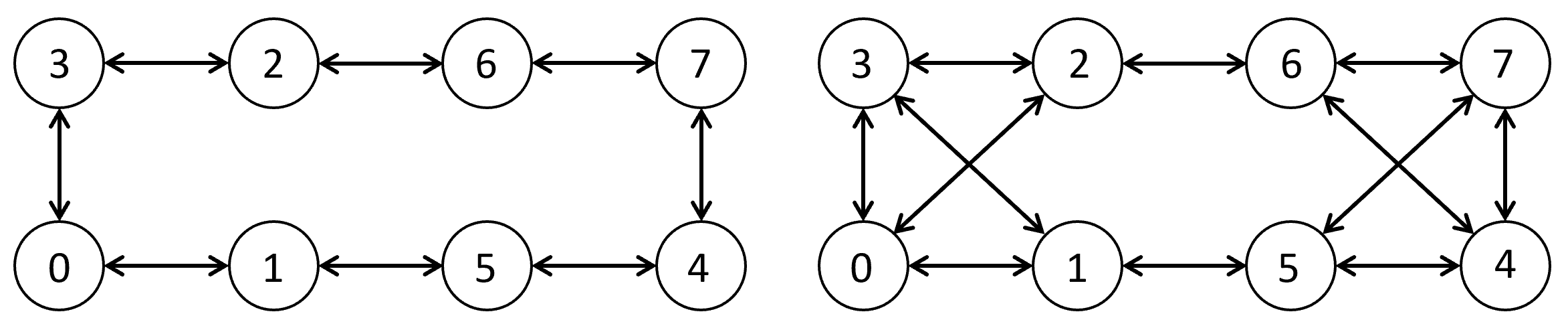} \caption{(Left) the L-Cayley graph $\LCay(G_{8},\{1,3\})$. (Right) the L-Cayley graph $\LCay(G_8,\{1,2,3\})$. } 
\label{F:G8S13andG8S123}
\end{center}
\end{figure}

\begin{table}
\begin{center}
\begin{tabular}{|c|cccccccc|}\hline
$\oplus$\ & 0 & 1 & 2& 3 & 4 & 5 & 6 & 7 \\\hline
0 & 0 & 1 & 2 & 3 & 4 & 5 & 6 & 7 \\
1 & 1 & 0 & 3 & 2 & 5 & 4 & 7 & 6 \\
2 & 2 & 3 & 0 & 1 & 6 & 7 & 4 & 5 \\
3 & 3 & 5 & 6 & 0 & 7 & 1 & 2 & 4 \\
4 & 4 & 2 & 1 & 7 & 0 & 6 & 5 & 3 \\
5 & 5 & 4 & 7 & 6 & 1 & 0 & 3 & 2 \\
6 & 6 & 7 & 4 & 5 & 2 & 3 & 0 & 1 \\
7 & 7 & 6 & 5 & 4 & 3 & 2 & 1 & 0 \\\hline
\end{tabular}\hskip1cm
\begin{tabular}{|c|cccccccc|}\hline
$\gyr$\ & 0 & 1 & 2& 3 & 4 & 5 & 6 & 7 \\\hline
0 & $I$ & $I$ & $I$ & $I$ & $I$ & $I$ & $I$ & $I$ \\
1 & $I$ & $I$ & $A$ & $A$ & $A$ & $A$ & $I$ & $I$ \\
2 & $I$ & $A$ & $I$ & $A$ & $A$ & $I$ & $A$ & $I$ \\
3 & $I$ & $A$ & $A$ & $I$ & $I$ & $A$ & $A$ & $I$ \\
4 & $I$ & $A$ & $A$ & $I$ & $I$ & $A$ & $A$ & $I$ \\
5 & $I$ & $A$ & $I$ & $A$ & $A$ & $I$ & $A$ & $I$ \\
6 & $I$ & $I$ & $A$ & $A$ & $A$ & $A$ & $I$ & $I$ \\
7 & $I$ & $I$ & $I$ & $I$ & $I$ & $I$ & $I$ & $I$ \\\hline
\end{tabular}
\caption{The gyroaddition table (left) and the gyration table (right) for $G_{8}=\{0,1,2,3,4,5,6,7\}$. The gyroautomorphism $A=(1~6)(2~5)$.}
\label{Ta:G8}
\end{center}
\end{table}
\end{Ex}

Having see Example~\ref{E:G8}, one may ask when $\LCay(G,S)$ is transitive. Some simple conditions are when $S$ is empty or $S$ contains a single order-2 element as proved in the following proposition. 
\begin{proposition}
Let $G$ be a gyrogroup and let $s$ be an element of $G$ such that $s=\ominus s$. Then the L-Cayley graph $\LCay(G,\{s\})$ is transitive.
\end{proposition}
\begin{proof}
Note that, from the assumption, $|G|$ is even by Lagrange's theorem for gyrogroups proved in~\cite[Theorem 5.7]{suksumran2014}. Since $\{s\}$ is symmetric and contains only one element, the Cayley graph is undirected by Theorem \ref{LAST31} and each vertex  has degree $1$. Hence the graph is a disjoint union of $|G|/2$ edges, which means that it is transitive.
\end{proof}
Continuing from the discussion at the end of Example~\ref{E:G8}, we provide the main theorem of this section.
\begin{theorem}
\label{T:Ltransitive} Let $(G,\oplus)$  be a gyrogroup and let $S$ be an symmetric subset of $G$. If $\gyr[g,s]$ is the identity map for all $g\in G$ and $s\in S$, then $\LCay(G,S)$ is transitive.
\end{theorem}
\begin{proof}
The idea of this theorem is that the condition on $\gyr[g,s]$ makes the right additions by any element of $G$ automorphisms on $\LCay(G,S)$. First, we note that since $S$ is symmetric, the Cayley graph is undirected. Now, let $u$ and $v$ be two vertices in $\LCay(G,s)$, i.e., two elements in $G$. Then  $v=u\oplus g$ for some $g\in G$. Suppose $w$ and $z$ are adjacent in $\LCay(G,S)$, that is, $w=s\oplus z$, for some $s\in S$. Adding $g$ on the right gives 
\[
w\oplus g=(s\oplus z)\oplus g=s\oplus (z\oplus\gyr[ z,s](g))=s\oplus(z\oplus g),
\]
which implies that $w\oplus g$ and $z\oplus g$ are adjacent. Hence the map $\phi:\LCay(G,S)\to\LCay(G,S)$ sending $x$ to $x\oplus g$ is an automorphism which maps $u$ to $v$. Therefore $\LCay(G,S)$ is transitive. \end{proof}
The converse of Theorem~\ref{T:Ltransitive} is not true as discussed in Example~\ref{E:G8} that some right additions of an element on a transitive Cayley graph ($\LCay(G_{8},\{1,3\})$) are not automorphisms on the graph. Below is an example for the gyrogroup $G_{16}$, a gyrogroup with $16$ elements.
\begin{Ex}
\label{E:G16}
Introduced in~\cite[p.41]{Ungar1}, the gyrogroup $G_{16}$ (called $K_{16}$ in the paper) has its addition and gyration tables as shown in Tables~\ref{Ta:G16addition} and~\ref{Ta:G16gyration}, respectively. Let $S=\{1,2,3\}$. From the gyration table, $\gyr[ g,s]$ is the identity map for all $g\in G$ and $s\in S$. By Theorem~\ref{T:Ltransitive}, the L-Cayley graph $\LCay(G_{16}, S)$ is transitive; it is a disjoint union of four copies of a complete graph with four vertices, as shown in Figure~\ref{F:G16S123}. Observe, for example, that the right addition by $1$ acts on $\LCay(G_{16},S)$ by flipping each copy of the complete graph exchanging the top and bottom pairs of vertices. Picking two vertices, say $1$ and $7=1\oplus 6$, an automorphism $\phi$ on the L-Cayley graph sending $1$ to $7$ is the right addition by $6$.\qed

\begin{figure}[ht]
\begin{center}
\includegraphics[height=5cm]{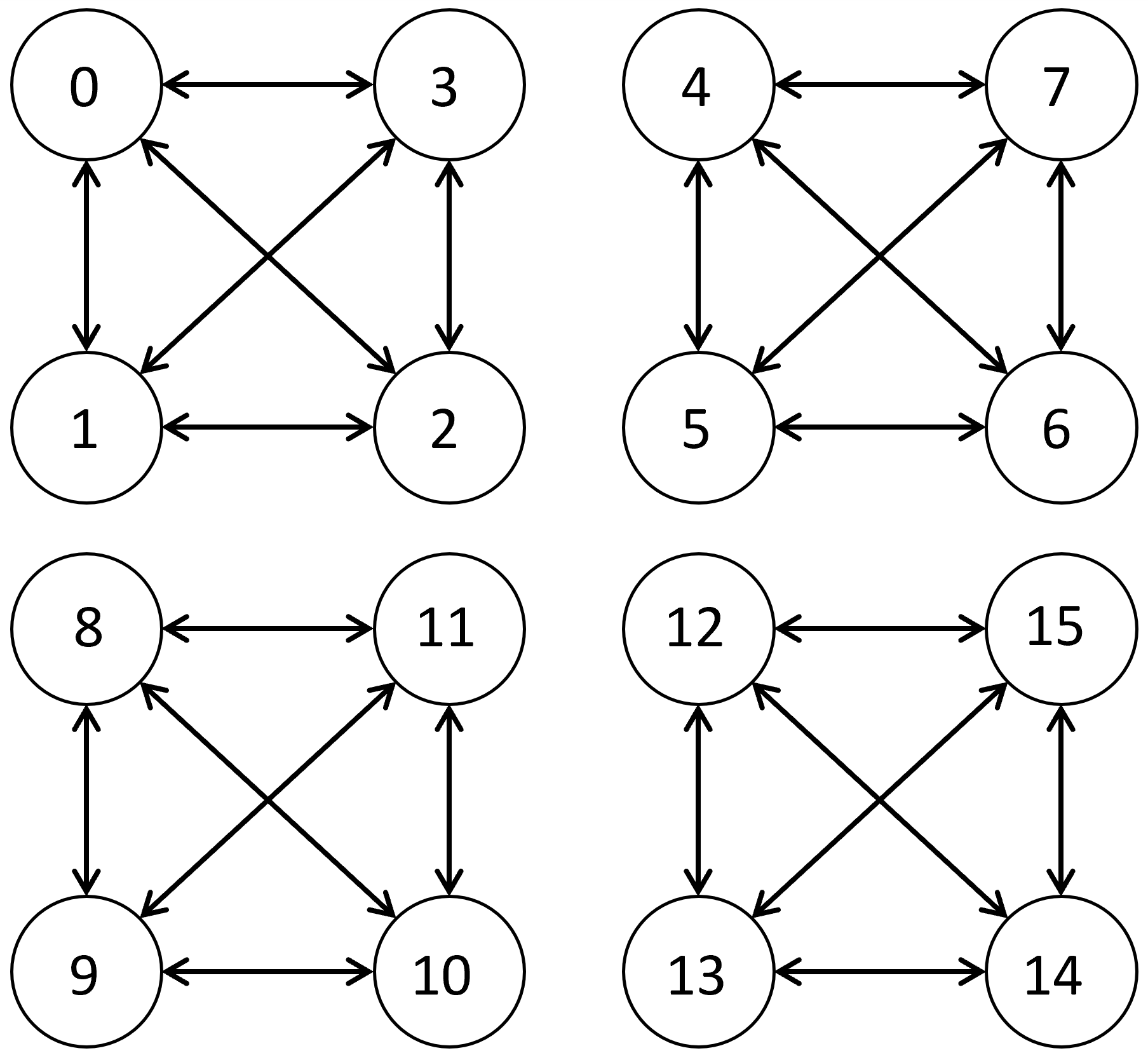} \caption{The L-Cayley graph $\LCay(G_{16},\{1,2,3\})$.} 
\label{F:G16S123}
\end{center}
\end{figure}
\end{Ex}

\begin{table}
\begin{center}
\begin{tabular}{|c|cccccccccccccccc|}
\hline $\oplus$ & 0 & 1 & 2 & 3 & 4 & 5 & 6 & 7 & 8 & 9 & 10 & 11 &
12 & 13 & 14 & 15 \\ \hline 0 & 0 & 1 & 2 & 3 & 4 & 5 & 6 & 7 & 8 &
9 & 10 &
11 & 12 & 13 & 14 & 15 \\
1 & 1 & 0 & 3 & 2 & 5 & 4 & 7 & 6 & 9 & 8 & 11 & 10 &
13 & 12 & 15 & 14 \\
2 & 2 & 3 & 1 & 0 & 6 & 7 & 5 & 4 & 11 & 10 & 8 & 9 &
15 & 14 & 12 & 13 \\
3 & 3 & 2 & 0 & 1 & 7 & 6 & 4 & 5 & 10 & 11 & 9 & 8 &
14 & 15 & 13 & 12 \\
4 & 4 & 5 & 6 & 7 & 3 & 2 & 0 & 1 & 15 & 14 & 12 & 13 &
9 & 8 & 11 & 10 \\
5 & 5 & 4 & 7 & 6 & 2 & 3 & 1 & 0 & 14 & 15 & 13 & 12 &
8 & 9 & 10 & 11 \\
6 & 6 & 7 & 5 & 4 & 0 & 1 & 2 & 3 & 13 & 12 & 15 & 14 &
10 & 11 & 9 & 8 \\
7 & 7 & 6 & 4 & 5 & 1 & 0 & 3 & 2 & 12 & 13 & 14 & 15 &
11 & 10 & 8 & 9 \\
8 & 8 & 9 & 10 & 11 & 12 & 13 & 14 & 15 & 0 & 1 & 2 & 3 &
4 & 5 & 6 & 7 \\
9 & 9 & 8 & 11 & 10 & 13 & 12 & 15 & 14 & 1 & 0 & 3 & 2 &
5 & 4 & 7 & 6 \\
10 & 10 & 11 & 9 & 8 & 14 & 15 & 13 & 12 & 3 & 2 & 0 & 1 &
7 & 6 & 4 & 5 \\
11 & 11 & 10 & 8 & 9 & 15 & 14 & 12 & 13 & 2 & 3 & 1 & 0 &
6 & 7 & 5 & 4 \\
12 & 12 & 13 & 14 & 15 & 11 & 10 & 8 & 9 & 6 & 7 & 5 & 4 &
0 & 1 & 2 & 3 \\
13 & 13 & 12 & 15 & 14 & 10 & 11 & 9 & 8 & 7 & 6 & 4 & 5 &
1 & 0 & 3 & 2 \\
14 & 14 & 15 & 13 & 12 & 8 & 9 & 10 & 11 & 4 & 5 & 6 & 7 &
3 & 2 & 0 & 1 \\
15 & 15 & 14 & 12 & 13 & 9 & 8 & 11 & 10 & 5 & 4 & 7 & 6 & 2 & 3 &
1& 0\\
\hline
\end{tabular}
\caption{The addition table of the gyrogroup $G_{16}$.}
\label{Ta:G16addition}
\end{center}
\end{table}

\begin{table}
\begin{center}
\begin{tabular}{|c|cccccccccccccccc|}
\hline $\textrm{gyr}$ & 0 & 1 & 2 & 3 & 4 & 5 & 6 & 7 & 8 & 9 & 10 &
11 & 12 & 13 & 14 & 15 \\ \hline 0 & $I$ & $I$ & $I$ & $I$ & $I$ &
$I$ & $I$ & $I$ & $I$ &
$I$ & $I$ & $I$ & $I$ & $I$ & $I$ & $I$ \\
1 & $I$ & $I$ & $I$ & $I$ & $I$ & $I$ & $I$ & $I$ & $I$ &
$I$ & $I$ & $I$ & $I$ & $I$ & $I$ & $I$ \\
2 & $I$ & $I$ & $I$ & $I$ & $I$ & $I$ & $I$ & $I$ & $I$ &
$I$ & $I$ & $I$ & $I$ & $I$ & $I$ & $I$ \\
3 & $I$ & $I$ & $I$ & $I$ & $I$ & $I$ & $I$ & $I$ & $I$ &
$I$ & $I$ & $I$ & $I$ & $I$ & $I$ & $I$ \\
4 & $I$ & $I$ & $I$ & $I$ & $I$ & $I$ & $I$ & $I$ & $A$ &
$A$ & $A$ & $A$ & $A$ & $A$ & $A$ & $A$ \\
5 & $I$ & $I$ & $I$ & $I$ & $I$ & $I$ & $I$ & $I$ & $A$ &
$A$ & $A$ & $A$ & $A$ & $A$ & $A$ & $A$ \\
6 & $I$ & $I$ & $I$ & $I$ & $I$ & $I$ & $I$ & $I$ & $A$ &
$A$ & $A$ & $A$ & $A$ & $A$ & $A$ & $A$ \\
7 & $I$ & $I$ & $I$ & $I$ & $I$ & $I$ & $I$ & $I$ & $A$ &
$A$ & $A$ & $A$ & $A$ & $A$ & $A$ & $A$ \\
8 & $I$ & $I$ & $I$ & $I$ & $A$ & $A$ & $A$ & $A$ & $I$ &
$I$ & $I$ & $I$ & $A$ & $A$ & $A$ & $A$ \\
9 & $I$ & $I$ & $I$ & $I$ & $A$ & $A$ & $A$ & $A$ & $I$ &
$I$ & $I$ & $I$ & $A$ & $A$ & $A$ & $A$ \\
10 & $I$ & $I$ & $I$ & $I$ & $A$ & $A$ & $A$ & $A$ & $I$ &
$I$ & $I$ & $I$ & $A$ & $A$ & $A$ & $A$ \\
11 & $I$ & $I$ & $I$ & $I$ & $A$ & $A$ & $A$ & $A$ & $I$ &
$I$ & $I$ & $I$ & $A$ & $A$ & $A$ & $A$ \\
12 & $I$ & $I$ & $I$ & $I$ & $A$ & $A$ & $A$ & $A$ & $A$ &
$A$ & $A$ & $A$ & $I$ & $I$ & $I$ & $I$ \\
13 & $I$ & $I$ & $I$ & $I$ & $A$ & $A$ & $A$ & $A$ & $A$ &
$A$ & $A$ & $A$ & $I$ & $I$ & $I$ & $I$ \\
14 & $I$ & $I$ & $I$ & $I$ & $A$ & $A$ & $A$ & $A$ & $A$ &
$A$ & $A$ & $A$ & $I$ & $I$ & $I$ & $I$ \\
15 & $I$ & $I$ & $I$ & $I$ & $A$ & $A$ & $A$ & $A$ & $A$ & $A$ & $A$
& $A$ & $I$ & $I$ & $I$ & $I$\\
\hline
\end{tabular}
\caption{The gyration table of the gyrogroup $G_{16}$. The gyroautomorphism $A=\{(8~9)(10~11)(12~13)(14~15)\}$.}
\label{Ta:G16gyration}
\end{center}
\end{table}

\section{Right-Cayley graphs}
\label{sec:Transitivity_R}
In this section, we provide  sufficient and necessary conditions for an R-Cayley graph of a gyrogroup to be undirected, give a sufficient condition for the Cayley graph to be transitive, and explore the connection of the cosets of L-subgyrogroups with the connected components of R-Cayley graphs. A few examples regarding these properties are also included.
\subsection{Undirectedness}
\label{ssec:Rundirectedness}
For an L-Cayley graph $\LCay(G,S)$, it is undirected when $S$ is symmetric. This is not the case for an R- Cayley graph as shown in the following examples.
\begin{Ex}
The right Cayley graph $\RCay(G_{16},\{8\})$ is not an undirected graph as shown in Figure~\ref{F:G16S8}. Note that $\{8\}$  right-generates a subgyrogroup $\{0,8\}$ which is not an L-subgyrogroup.
\qed  

\begin{figure}[ht]
\begin{center}
\includegraphics[height=2.7cm]{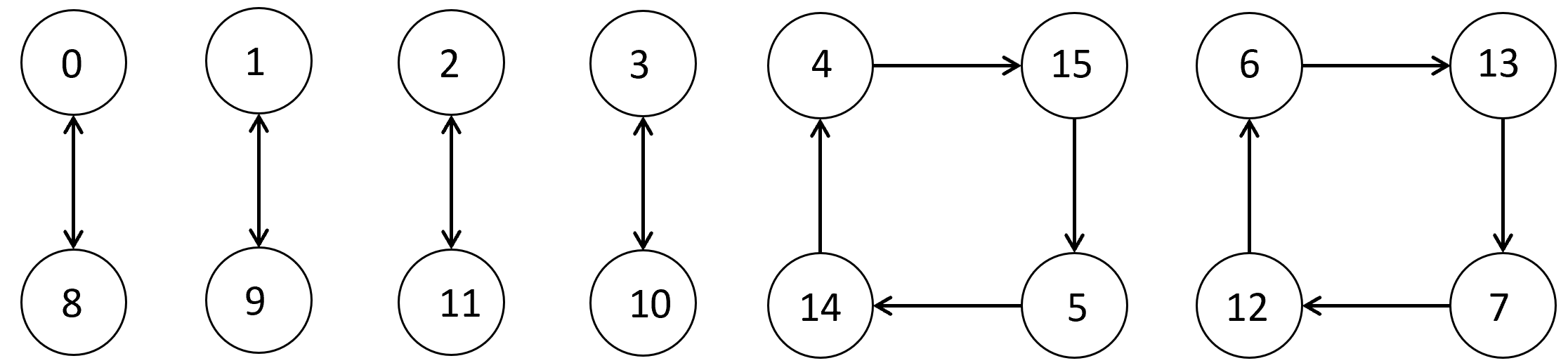} \caption{The R-Cayley graph $\RCay(G_{16},\{8\})$.} 
\label{F:G16S8}
\end{center}
\end{figure}
\end{Ex}

\begin{Ex}
The right Cayley graph $\RCay(G_{16},\{1,8\})$ is not an undirected graph as shown in Figure~\ref{F:G16S18}. Note that $\{1,8\}$ right-generates an L-subgyrogroup $\{0,1,8,9\}$.
\qed

\begin{figure}[ht]
\begin{center}
\includegraphics[height=2.7cm]{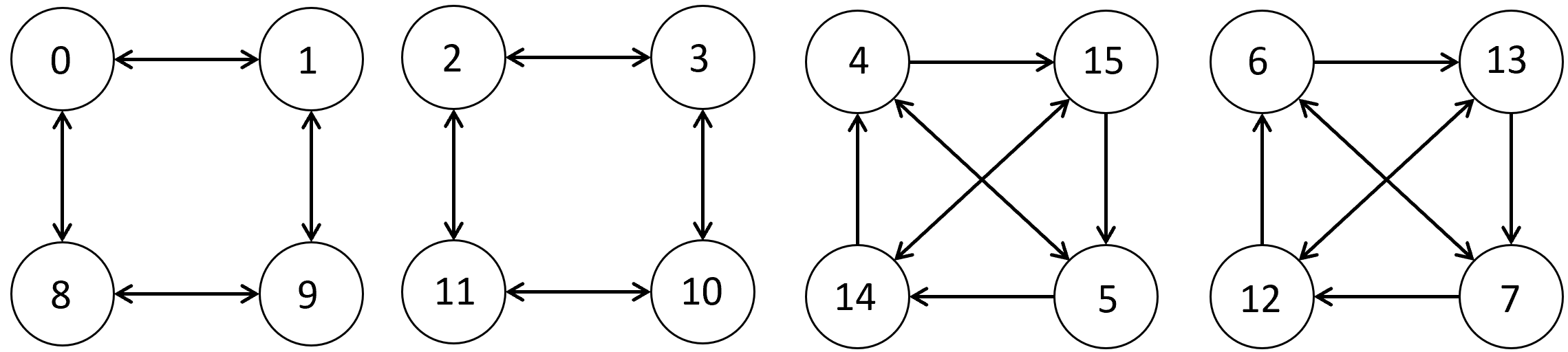} \caption{The R-Cayley graph $\RCay(G_{16},\{1,8\})$.} \label{F:G16S18}
\end{center}
\end{figure}
\end{Ex}

The following theorem gives  sufficient and necessary conditions for an L-Cayley graph of a gyrogroup to be undirected.

\begin{theorem}
\label{T:Rundirected}
Let $(G,\oplus)$  be a gyrogroup and let $S$ be a symmetric subset of $G$. If $\gyr[g,s](S)=S$ for all $s\in S$ and $g\in G$ then $\RCay(G,S)$ is undirected. Conversely, if $\RCay(G,S)$ is undirected then $\gyr[g,s]s\in S$, for all $g\in G$ and $s\in S$.
\end{theorem}
\begin{proof}
Suppose there is a directed edge from  $u$ to $v$  in $\RCay(G,S)$. Then $v=u\oplus s$ for some $s\in S$. By the right cancellation law and the definition of coaddition, we have $u=v\boxminus s=v\oplus \gyr[v,s](\ominus s)=v\oplus s'$, for some $s'\in S$, which implies that there is a directed edge from $v$ to $u$. Hence $\RCay(G,S)$ is undirected.

Now suppose that $\RCay(G,S)$ is undirected. Given elements $g\in G$ and $s\in S$. Then $g\boxminus s$ is a vertex in the Cayley graph. Since $(g\boxminus s)\oplus s=g$, there is a directed edge $g\boxminus s\to g$. The undirected condition implies the existence of a directed edge $g\to g\boxminus s $, which means that $g\oplus s'=g\boxminus s=g\oplus \gyr[g,s](\ominus s)$ for some $s'\in S$. Using the general left cancellation law and Item~\ref{Id:inverse through gyro} in Theorem~\ref{T:identities}, we have $\gyr[g,s]s=\ominus s'$ which is an element in $S$ as desired.
\end{proof}

\begin{Ex}
\label{E:G16S89}
Since $\gyr[g_1,g_2](\{8,9\})=\{8,9\}$ for all $g_1,g_2\in G_{16}$, the right Cayley graph $\RCay(G_{16},\{8,9\})$ is an undirected graph as shown in Figure~\ref{F:G16S89}. Observe, for instance, that $14\to 4$ because $4=14\oplus 8$, while $4\to 14$ because $14=4\oplus 9$ but $8$ and $9$ are not inverses of one another. This is different from the left Cayley graph where a bidirected edge arises from an element of $S$ and its inverse. Note that $\{8,9\}$ right-generate an L-subgyrogroup $\{0,1,8,9\}$.
\qed

\begin{figure}[ht]
\begin{center}
\includegraphics[height=2.7cm]{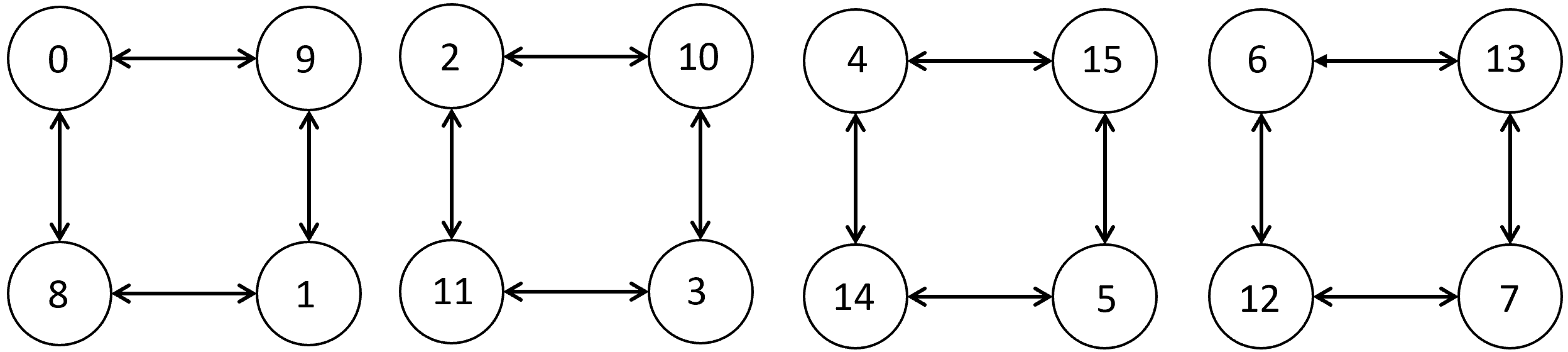} \caption{The R-Cayley graph $\RCay(G_{16},\{8,9\})$.} \label{F:G16S89}
\end{center}
\end{figure}
\end{Ex}

\subsection{Transitivity}
\label{ssec:Rtransitivity}
In this subsection, we give a sufficient condition for an R-Cayley graph to be transitive.
\begin{theorem}
\label{T:Rtransitive}
Let $(G,\oplus)$  be a gyrogroup and let $S$ be a symmetric subset of $G$ such that $\gyr[ g,g'](S)=S$ for all $g,g'\in G$. Then $\RCay(G,S)$ is transitive.
\end{theorem}
\begin{proof}
We note that, by Theorem~\ref{T:Rundirected}, $\RCay(G,S)$ is undirected. A similar idea as in Theorem~\ref{T:Ltransitive} is applied here; the condition on gyrators causes the left additions by any element of $G$ to be  automorphisms on $\RCay(G,S)$. Let $u$ and $v$ be two vertices in $\RCay(G,S)$, i.e., two elements in $G$. Then  $v=g\oplus u$ for some $g\in G$. Suppose $w$ and $z$ are adjacent in $\RCay(G,S)$, that is, $w=z\oplus s$, for some $s\in S$. Adding $g$ on the left yields 
\[
g\oplus w=g\oplus(z\oplus s)=(g\oplus z)\oplus \gyr[g,z](s)=(g\oplus z)\oplus s',
\]
for some $s'\in S$. Hence $g\oplus w$ and $g\oplus z$ are adjacent.
This implies that the map $\phi:x\mapsto g\oplus x$ is an automorphism on $\RCay(G,S)$ sending $u$ to $v$. Therefore $\RCay(G,S)$ is transitive. \end{proof}
The right Cayley graph $\RCay(G_{16},\{8,9\})$ in Example~\ref{E:G16S89} is transitive since $\{8,9\}$ satisfies the gyration condition in Theorem~\ref{T:Rtransitive}. We present other two examples as follows.
\begin{Ex}
\label{E:RCayG16S123}The subset $S=\{1,2,3\}$ has a property that for any $g,g'\in G_{16}$ and any $s\in S$, $\gyr[g,g'](s)=I(s)=s$, which means that it has the required gyrator condition stated in Theorem~\ref{T:Rtransitive}. The R-Cayley graph $\RCay(G_{16},S)$ is isomorphic to $\LCay(G_{16},S)$ shown in Figure~\ref{F:G16S123}, with the same vertex labeling, but different edge labeling. 

\end{Ex}

\begin{Ex}
\label{E:RCayG16S891011}
Consider a subset $S'=\{8,9,10,11\}$ of $G_{16}$. Since the non-identity automorphism $A$ swaps 8 with 9 and 10 with 11, and these four elements are self-inverse, the subset $S$ has the required gyrator condition-$\gyr[g,g'](S)=S$ for all $g,g'\in G_{16}$ and is symmetric. The R-Cayley graph $\RCay(G_{16},S)$ is shown in Figure~\ref{F:G16S891011}. It is easily seen to be transitive. The left addition by $15$ is an automorphism on the graph exchanging the two connected components, and swapping the inner and outer cycles, in particular, it sends $5$ to $8$, $15$ to $0$, and $4$ to $9$. 
\end{Ex}

\begin{figure}[ht]
\begin{center}
\includegraphics[height=5cm]{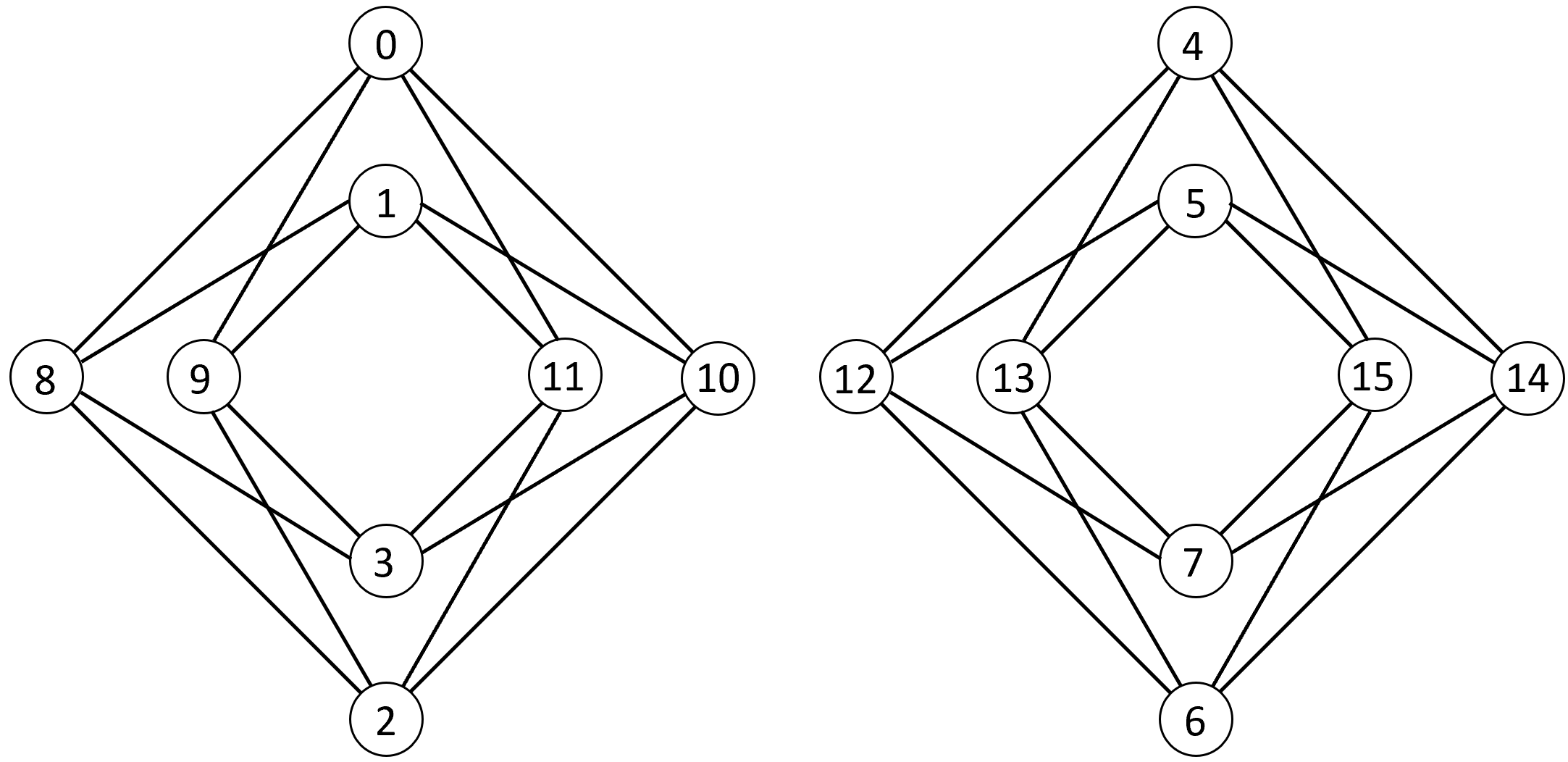} \caption{The R-Cayley graph $\RCay(G_{16},\{8,9,10,11\})$.} \label{F:G16S891011}
\end{center}
\end{figure}

\subsection{Connectedness}
\label{ssec:Rconnectedness}
In this subsection, we show a relationship between the cosets of L-subgyrogroups and the connected components of R-Cayley graphs.

\begin{theorem}
Let $G$ be a gyrogroup and let $S$ be a symmetric subset of $G$ such that it right-generates an L-subgyrogroup $H$ and $\gyr[g,h](S)=S$ for all $g\in G$ and $h\in H$. Then two vertices $u$ and $v$ are in the same connected component of $\RCay(G,S)$ if and only if $u$ and $v$ are in the same left-coset of $H$.
\end{theorem}
\begin{proof}
First, note that from Theorem~\ref{T:Rundirected}, $\RCay(G,S)$ is undirected. Suppose that $u$ and $v$ are two vertices in the same connected component. Then there is a path connecting $u$ and $v$, that is, 
\[
u=((\cdots((v\oplus s_1)\oplus s_2)\oplus\cdots\oplus s_{n-2})\oplus s_{n-1})\oplus s_n,
\]
for some $s_1,...,s_n\in S$.
Keep moving the parentheses to the right starting from the outer most one yields
\begin{align*}
u&=((\cdots((v\oplus s_1)\oplus s_2)\oplus\cdots\oplus s_{n-2})\oplus s_{n-1})\oplus s_n\\
&=(\cdots((v\oplus s_1)\oplus s_2)\oplus\cdots\oplus s_{n-2})\oplus (s_{n-1}\oplus \gyr[s_{n-1},g] s_n)\\
&=(\cdots((v\oplus s_1)\oplus s_2)\oplus\cdots\oplus s_{n-2})\oplus(s_{n-1}\oplus s_n')\\
&=(\cdots((v\oplus s_1)\oplus s_2)\oplus\cdots\oplus s_{n-2})\oplus h_1\\
&~~\vdots\\
&=v\oplus h_{n-1},
\end{align*}
where $g=(\cdots((v\oplus s_1)\oplus s_2)\oplus\cdots)\oplus s_{n-2}, s_n'=\gyr[s_{n-1},g] s_n\in S$, $h_1=s_{n-1}\oplus s_n'\in H$, and $h_{n-1}\in H$. Hence $u$ and $v$ are in the same $H$-coset.

Conversely, suppose $u$ and $v$ are in the same left-coset. Then $\ominus v\oplus u\in H$ and
\[
u=v\oplus((\cdots(s_1\oplus s_2)\oplus\cdots\oplus s_{n-1})\oplus s_n),
\]
for some $s_1,...,s_n\in S$.
Now, keep moving the parentheses to the left starting from the outer most one yields

\begin{align*}
u&=v\oplus((\cdots(s_1\oplus s_2)\oplus\cdots\oplus s_{n-1})\oplus s_n)\\
&=(v\oplus(\cdots(s_1\oplus s_2)\oplus\cdots\oplus s_{n-1}))\oplus \gyr[v,h_{1}]s_n\\
&=(v\oplus(\cdots(s_1\oplus s_2)\oplus\cdots\oplus s_{n-1}))\oplus s_n'\\
&~~\vdots\\
&=(\cdots((v\oplus s_1)\oplus s_2')\oplus\cdots\oplus s_{n-1}')\oplus s_n', 
\end{align*}
where $h_1=(\cdots(s_1\oplus s_2)\oplus\cdots)\oplus s_{n-1}\in H, s_n'=\gyr[v,h_1]s_n\in S,$ and $s_2',...,s_{n-1}'\in S$. Hence $u$ and $v$ are in the same connected component in the Cayley graph.
\end{proof}
Examples~\ref{E:G16S89} and~\ref{E:RCayG16S891011} are examples of this theorem. We have four left-cosets and hence four connected component in Example~\ref{E:G16S89}, whereas there are two left-cosets represented by two connected components in Example~\ref{E:RCayG16S891011}.

\section*{Acknowledgements} The work of T. Suksumran was supported by Chiang Mai University.

\bibliographystyle{amsplain}\addcontentsline{toc}{section}{References}
\bibliography{References}

\noindent
\footnotesize{T. Suksumran\\
Department of Mathematics, Faculty of Science, Chiang Mai University, Chiang Mai, Thailand
\\
e-mail: teerapong.suksumran@cmu.ac.th\\
\\
P. Khachorncharoenkul, R. Maungchang and K. Prathom\\
School of Science, Walailak University, Nakhon Si Thammarat, Thailand\\
e-mail: prathomjit.kh@mail.wu.ac.th, mate105@gmail.com, kiattisak.pr@mail.wu.ac.th}
\end{document}